\documentclass[12pt,a4paper,oneside]{amsart}

\usepackage{amsfonts, amsmath, amssymb, amsthm, amscd, hyperref}
\usepackage{anysize}
\usepackage{graphicx}
\usepackage{pdfsync}
\usepackage[all]{xy}
\usepackage{xcolor}
\usepackage{mathabx}

\newtheorem{theorem}{Theorem}[section]
\newtheorem{lemma}[theorem]{Lemma}

\newtheorem{corollary}[theorem]{Corollary}
\newtheorem{proposition}[theorem]{Proposition}

\theoremstyle{definition}
\newtheorem{definition}[theorem]{Definition}

\theoremstyle{remark}
\newtheorem{remark}[theorem]{Remark}
\newtheorem{question}[theorem]{Question}

\numberwithin{equation}{section}

\DeclareMathOperator{\diam}{diam}

\DeclareMathOperator{\rk}{rk}
\DeclareMathOperator{\dist}{dist}
\DeclareMathOperator{\interior}{int}

\DeclareMathOperator{\UW}{UW}

\newcommand{\eps}{\varepsilon}

\begin{document}

\title{Local-to-global Urysohn width estimates}

\author{Alexey~Balitskiy{$^\clubsuit$}}

\email{{$^\clubsuit$}balitski@mit.edu}

\author{Aleksandr~Berdnikov{$^\spadesuit$}}

\email{{$^\spadesuit$}aberdnik@mit.edu}

\address{{$^\clubsuit$}{$^\spadesuit$} Dept. of Mathematics, Massachusetts Institute of Technology, 182 Memorial Dr., Cambridge, MA 02142, USA}
\address{{$^\clubsuit$} Institute for Information Transmission Problems RAS, Bolshoy Karetny per. 19, Moscow, Russia 127994}






\begin{abstract}
The notion of the Urysohn $d$-width measures to what extent a metric space can be approximated by a $d$-dimensional simplicial complex. We investigate how local Urysohn width bounds on a Riemannian manifold affect its global width. We bound the $1$-width of a Riemannian manifold in terms of its first homology and the supremal width of its unit balls. Answering a question of Larry~Guth, we give examples of $n$-manifolds of considerable $(n-1)$-width in which all unit balls have arbitrarily small $1$-width. We also give examples of topologically simple manifolds that are locally nearly low-dimensional.
\end{abstract}

\maketitle

\section{Introduction}
\label{sec:intro}

In the paper~\cite{guth2017volumes} Larry Guth proves that, on a closed Riemannian manifold, local volume estimates translate into global information about the Urysohn width. This resolved a conjecture of Gromov~\cite{gromov1983filling}, and provided an alternative way to prove the celebrated systolic inequality of Gromov. Guth also conjectured a generalization of his theorem, dealing with the Hausdorff content on compact metric spaces in place of volume, and his conjecture was established by Liokumovich, Lishak, Nabutovsky, and Rotman~\cite{liokumovich2019filling}. Shortly after that, a simple and clever proof was given by Panos Papasoglu~\cite{papasoglu2020uryson}, and the method employed there gives the simplest and cleanest proof~\cite{nabutovsky2019linear} of Gromov's systolic inequality, with the best dimensional constants known so far.

The notion of the Urysohn $d$-width, popularized by Gromov \cite{gromov1983filling, gromov1988width}, is a metric invariant measuring to what extent a metric space can be approximated by a $d$-dimensional simplicial complex. After several successful applications, as in the systolic inequality, it became an invariant of independent interest. The original definition (equivalent to the one we give below) by Pavel~Urysohn was given in 1920s (and published posthumously by Pavel~Alexandrov~\cite{alexandroff1926notes}) in terms of closed coverings of bounded multiplicity. We use a different but equivalent definition.

\begin{definition}
\label{def:urysohn}
The \emph{Urysohn $d$-width} of a compact metric space $X$ is
$$
\UW_d(X) = \inf\limits_{\pi: X \to Y} \sup\limits_{y \in Y} \diam(\pi^{-1}(y)),
$$
where the infimum is taken over all continuous maps $\pi$ from $X$ to any simplicial complex $Y$ of dimension at most $d$. (Recall that $\diam A = \sup\limits_{a,a' \in A} \dist_X(a,a')$.)
\end{definition}

The width obeys the following trivial properties.
\begin{itemize}
  \item It is monotone with respect to inclusion: for any closed subset $S \subset X$, $\UW_d(S) \le \UW_d(X)$ for all $d$. Here and everywhere the width of a closed subset $S \subset X$ is defined using the extrinsic metric induced by $X$, for measuring diameters.
  \item It is monotone in dimension: $\UW_0(X) \ge \UW_1(X) \ge \UW_{2}(X) \ge \ldots$.
  \item The $n$-width of a closed Riemannian $n$-manifold is zero.
  \item The $(n-1)$-width of a closed Riemannian $n$-manifold is greater than zero, as follows from the Lebesgue covering lemma~\ref{lem:lebesgue}.
\end{itemize} 

In the same paper~\cite{guth2017volumes}, Guth gives an example of a metric on $S^3$ with locally small but globally large $2$-width~\cite[Section~4]{guth2017volumes}. Further, he asks if there is a setting in which local Urysohn width bounds translate into global ones. 

\begin{question}[{\cite[Question~5.3]{guth2017volumes}}]
\label{ques:larry}
Suppose that $M^n$ is a Riemannian manifold such that each unit ball $B\subset M$ has $\UW_q(B) < \eps$. If $\eps$ is sufficiently small, does this inequality imply anything about $\UW_{q'}(M)$ for some $q' \ge q$?
\end{question}

We answer this question in the negative (see Theorem~\ref{thm:tubes} below), and investigate how additional topological complexity assumptions affect the answer.

Our first result is an estimate of $1$-width of a closed Riemannian manifold $M$, depending on its topological complexity as well as the supremal width of its unit balls.

\begin{theorem}
\label{thm:surface}
Let $M^n$ be a closed Riemannian manifold with the first $\mathbb{Z}/2$-Betti number $\beta = \rk H_1(M; \mathbb{Z}/2)$. If every unit ball has $1$-width less than $1/15$, then $\UW_1(M) < \beta+1$.
\end{theorem}

The dependence on $\beta$ does not seem optimal. The best example we know has $\UW_1(M) \sim \beta^{1/n}$ (see Figure~\ref{fig:tubes}). This example is constructed in our second theorem, which resolves Guth's question in the negative.

\begin{theorem}
\label{thm:tubes}
For any $\eps>0$, there exists a closed Riemannian manifold $M^n$ with all unit balls of $1$-width less than $\eps$, and such that $\UW_{n-1}(M) \ge 1$.
\end{theorem}

\begin{figure}[ht]
  \centering
  \includegraphics[width=0.8\textwidth]{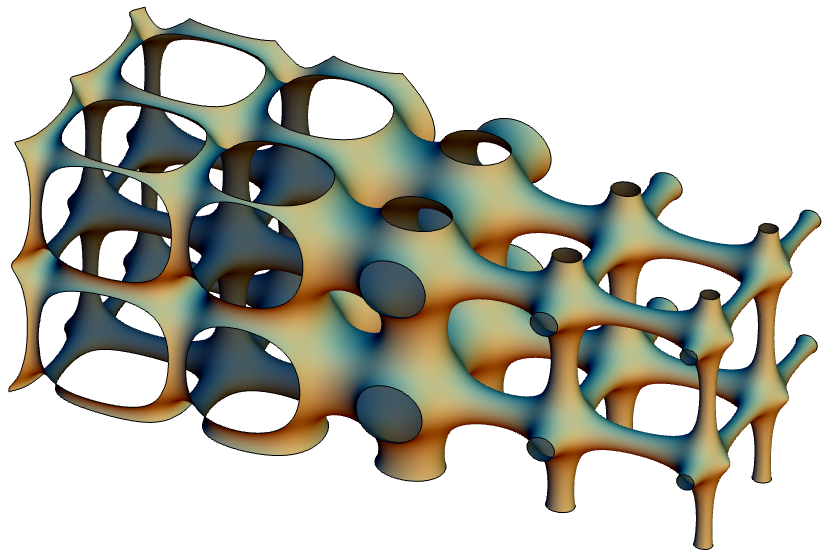}
  \caption{A piece of the surface from Theorem~\ref{thm:tubes} for $n=2$. The whole surface is made by replicating this piece periodically many times and closing up the ends. Roughly speaking, the left half of this surface has small Urysohn $1$-width, as well as the right half, while the whole surface has large Urysohn 1-width}
  \label{fig:tubes}
\end{figure}

Note that the negative result with $q=1$ and $q'=n-1$ is the strongest possible over all choices of $q, q'$. Therefore, the answer to Question~\ref{ques:larry} is negative for all $q, q'$. 

The example establishing this theorem has large Betti numbers. If one is looking for a topologically simple example, our third result gives it with $M^n$ being a ball (but with a worse dimension in the local width bound).

\begin{theorem}
\label{thm:ball}
For any $\eps>0$, there is a metric on the $n$-ball $M^n$ (or $n$-sphere, or $n$-torus) such that its $(n-1)$-width is at least $1$ but $\UW_{\lceil \log_2 (n+1) \rceil}(B) < \eps$ for every unit ball $B \subset M$.
\end{theorem}

\subsection*{Acknowledgements} We are grateful to Larry~Guth for numerous conversations and his remarks on this paper. We also thank Hannah~Alpert and Panos~Papasoglu for the stimulating discussions that led us to these questions.

\section{Bounding width from below}

Before we get to the main results, let us discuss the main tools one can use to show a space has substantial Urysohn width.

The \textbf{first tool} is the Lebesgue covering lemma (discovered by Lebesgue~\cite{lebesgue1911non} and first proved by Brouwer~\cite{brouwer1913naturlichen}), which can be used to show that the $(n-1)$-width of the unit Euclidean $n$-cube equals $1$.

\begin{lemma}
\label{lem:lebesgue}
Every continuous map $f: [0,1]^n \to Y^d$ from the unit $n$-cube to an $d$-dimensional simplicial complex, $d<n$, has a fiber $f^{-1}(y)$ meeting some two opposite facets of the cube.
\end{lemma}

The \textbf{second tool} amounts to the ``fiber contraction'' argument, which goes back to Gromov~\cite[Proposition~(F$_1$)]{gromov1988width}. A detailed exposition can be found in~\cite[Section~5]{guth2005lipshitz}. We quote here a version of this argument due to Guth~\cite[Lemma~5.2]{guth2005lipshitz}.

\begin{lemma}
\label{lem:contractfibers}
Let $W$ be a Riemannian manifold of convexity radius at least $\rho$; that is, any two points in a ball of radius $<\rho$ are connected by a unique minimal geodesic within this ball. Let $\pi: X \to Y$ be a map from a metric space $X$ to a simplicial complex $Y$, such that all fibers of $\pi$ have diameter less than $\rho$. Then any $1$-Lipschitz map $f: X \to W$ is homotopic to a map factoring as $g \circ \pi$, for some $g: Y \to W$. Moreover, the homotopy moves each point of $W$ by less than $2\rho$.
\end{lemma}

\begin{corollary}
\label{cor:secondtool}
Let $f: X \to W$ be a $1$-Lipschitz map from a metric space $X$ to a Riemannian manifold $W$ of convexity radius at least $\rho$. Suppose one of the following conditions holds.
\begin{enumerate}
  \item The induced map $f_* : H_n(X) \to H_n(W)$ is non-trivial for some $n$.
  \item For some closed subsets $X_0 \subset X$, $W_0 \subset W$, $f$ sends $X_0$ to $W_0$, and the induced map $f_* : H_n(X, X_0) \to H_n(W, U_{2\rho}(W_0))$ is non-trivial for some $n$. (Here $U_{2\rho}(W_0)$ is the neighborhood of $W_0$ of radius $2\rho$.)
\end{enumerate}
Then $\UW_{n-1}(X) \ge \rho$.
\end{corollary}

\section{Surface width estimates}
\label{sec:lowdim}

\begin{proof}[Proof of Theorem~\ref{thm:surface}]
This proof follows closely the ideas from~\cite[Section~1]{guth2005lipshitz},~\cite[Appendix~1,~(E$_1$)-(E$_1''$)]{gromov1983filling}. The main theorem in~\cite[Section~1]{guth2005lipshitz} says, basically, that in the case $M \simeq S^2$, there is a universal way to measure the Urysohn 1-width: it is given by the map to the set of the connected components of distance spheres around any point. The largest diameter of such a component gives the value $\UW_1(S^2)$ within a factor of $7$. We adapt this idea to higher dimensions, taking into account the topological complexity as well.

Pick any point $p \in M^n$. Consider the distance spheres $S_r(p)$. We show that $\UW_0(S_r(p)) < \beta + 1$ for each $r$. For $r < 1/2$ this is clear, so fix $r \ge 1/2$ and suppose that $\UW_0(S_r(p)) \ge \beta + 1$, so there are points $x$ and $y$ distance $\beta+1$ apart in the same connected component of $S_r(p)$. Denote by $\gamma$ a curve connecting $x$ and $y$ inside $S_r(p)$ (we can assume it exists by perturbing slightly the distance function $\dist(\cdot, p)$). Denote $x_0 = x$, $x_{\beta+1} = y$, and pick points $x_k \in \gamma$, $1 \le k \le \beta$, so that $\dist(x, x_k) = k$. Notice that $\dist(x_i, x_j) \ge |i-j|$. Denote by $g_k$ a minimal geodesic from $p$ to $x_k$, for $0 \le k \le \beta+1$. Denote by $\ell_k$, $0 \le k \le \beta$, the loop formed by the curves $g_{k}$, $g_{k+1}$ and the part of $\gamma$ between $x_{k}$ and $x_{k+1}$. The loops $\ell_0, \ldots, \ell_{\beta}$ cannot be independent in $H_1(M; \mathbb{Z}/2)$; hence, there exist indices $0 \le i_1 < \ldots < i_r \le \beta$ such that $[\ell_{i_1}] + \ldots + [\ell_{i_r}] = 0$ in $H_1(M; \mathbb{Z}/2)$.

The concatenation of $\ell_{i_1}, \ldots, \ell_{i_r}$ bounds a $2$-chain $D$ in $M$, which we also view as a closed subset of $M$. Assuming that $1$ is a regular value of $\dist(\cdot, x_{i_1})$ on $D$ (otherwise perturb this function slightly), one can view the intersection $D' = D \cap B_1(x_{i_1})$ as a $2$-chain as well. Now consider the map $f: M \to \mathbb{R}^2$ given by $f(\cdot) = (\dist(\cdot, p), \dist(\cdot, x_{i_1}))$ (see Figure~\ref{fig:surface}). Note it is $\sqrt{2}$-Lipschitz. We will show that $f(D')$ covers a disk $O$ of radius $\frac{\sqrt{2}-1}{2}$ in $\mathbb{R}^2$; formally speaking, the map $f : (D', \partial D') \to (\mathbb{R}^2, \mathbb{R}^2 \setminus \interior O)$ is of degree $1 \pmod{2}$. Then Corollary~\ref{cor:secondtool} can be applied to $f$ (composed with a $1/\sqrt{2}$-homothety, to make it $1$-Lipschitz), implying $\UW_1(D') \ge \frac{\sqrt{2}-1}{4\sqrt{2}}$. On the other hand, $\UW_1(D') \le \UW_1(B_1(x_{i_1})) < 1/15$ since $D' \subset B_1(x_{i_1})$, which gives a contradiction.

\begin{figure}[ht]
  \centering
  \includegraphics[width=1.0\textwidth]{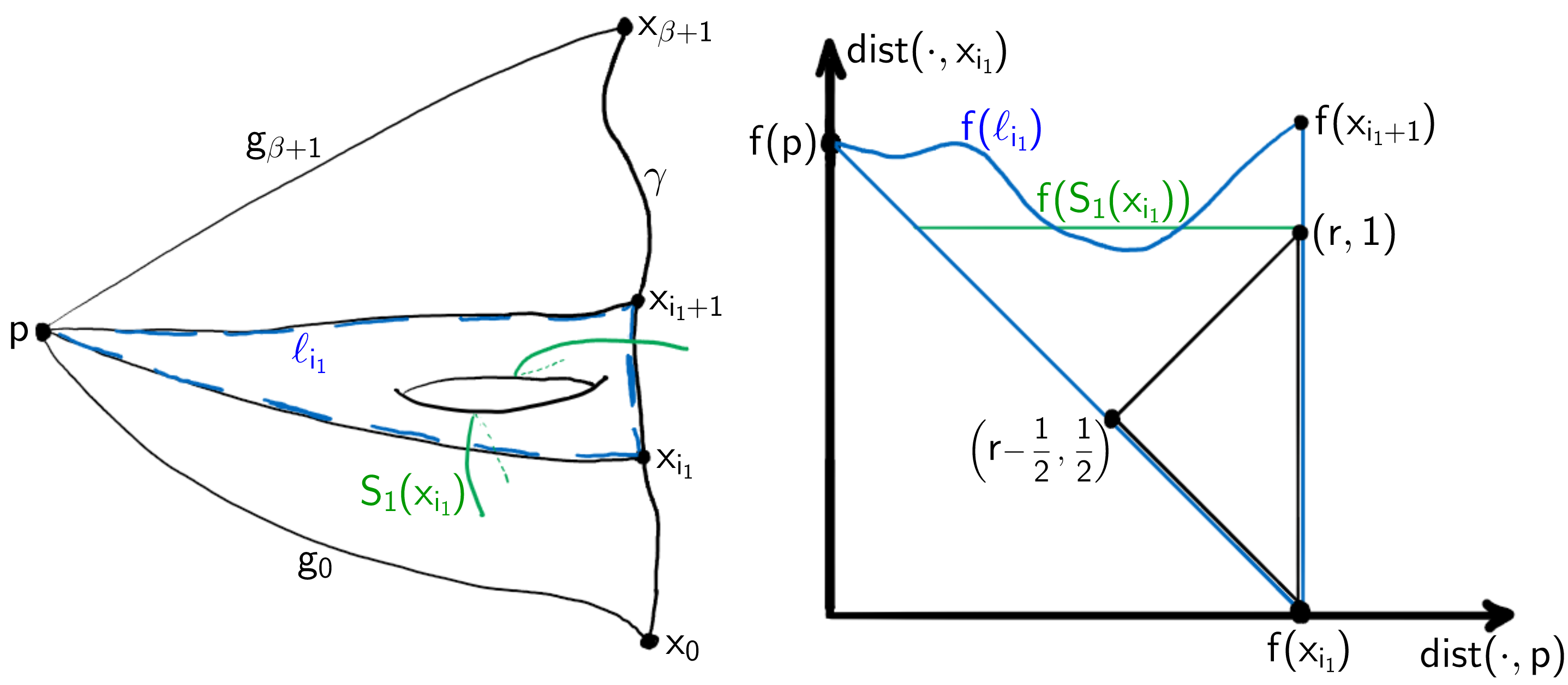}
  \caption{The map $f$ covers a substantial triangular region}
  \label{fig:surface}
\end{figure}

\textbf{Observation 1.} The image of $f$ lies above the straight line $\lambda$ connecting points $(r,0)$ and $(0,r)$.

\textbf{Observation 2.} Consider the triangle $\Delta$ with the vertices $(r,0)$, $(r,1)$, and $(r-1/2,1/2)$, and observe that $f(\partial D') \cap \interior \Delta = \varnothing$. Indeed, $\partial D' \subset \partial D \cup S_1(x_{i_1})$, so the image $f(\partial D')$ is contained in the union of the following curves:
\begin{itemize}
  \item the line $\lambda$, where $f(g_{i_1})$ lies;
  \item the vertical straight line through $(r,0)$, where $f(\gamma)$ lies;
  \item the horizontal line through $(r,1)$, where $f(S_1(x_{i_1}))$ lies;
  \item the curves $f(g_{k})$, $k > i_1$, each of which can be viewed as the graph of a $1$-Lipschitz function of argument $\dist(\cdot, p)$; they all lie above the straight line connecting $(r,1)$ and $(r-1/2,1/2)$.
\end{itemize}

\textbf{Observation 3.} Let $q = g_{i_1} \cap S_{1/2}(x_{i_1})$, and observe that the geodesic segment $[q,x_{i_1}] \subset g_{i_1}$ is present in the $1$-chain $\partial D'$. The image $f([q,x_{i_1}])$ is the straight line segment between $(r,0)$ and $(r-1/2,1/2)$ (traversed once). Other parts of $f(\partial D')$ are all contained in the union $f(\gamma) \cup f([p,q]) \cup f(S_1(x_{i_1})) \cup \bigcup\limits_{k> i_1} f(g_{k})$, avoiding this straight line segment. In view of the previous two observations, $f(\partial D')$ winds around $\Delta$ nontrivially. Therefore, the degree of $f : (D', \partial D') \to (\mathbb{R}^2, \mathbb{R}^2 \setminus \interior \Delta)$ is $1 \pmod{2}$.

\textbf{Observation 4.} The disk $O$ inscribed in $\Delta$ is of radius $\frac{\sqrt{2}-1}{2}$.

This concludes the proof.
\end{proof}

\begin{remark}
  Under the same assumptions (every unit ball in the surface $M$ has $1$-width less than $1/15$), one can show that the \emph{homological systole} (the length of the shortest loop that is not null-homologous) is less than $2$, regardless of genus. One way to show it is to adjust the proof of~\cite[Theorem~4.1]{guth2005lipshitz}.
\end{remark}

\section{Manifolds of small local but large global width}
\label{sec:highdim}

The constructions of this section are inspired by the mother of examples~\cite[Example~H$_1''$]{gromov1988width}.

\subsection{Local join representation}
A crucial ingredient for the constructions below is a decomposition of $\mathbb{R}^{2n-1}$ as the ``local join'' of several $1$-dimensional complexes.
\begin{lemma}
\label{lem:triangulation}
  Fix $0 < \eps < 1$. It is possible to triangulate $\mathbb{R}^{2n-1}$ by simplices with the following properties:
  \begin{enumerate}
    \item Each simplex is $c_n$-bi-Lipschitz to a regular simplex with edge length $\eps$.
    \item The vertices of triangulation can be colored by colors $1$ through $2n$ so that each simplex receives all distinct colors.
    \item The colored triangulation can be taken periodic with respect to $n$ almost orthogonal translation vectors of length $\approx 10 c_n$; hence the colored triangulation descends to the product $T^n \times \mathbb{R}^{n-1}$, where $T^n$ is a flat torus with convexity radius at least $1$.
  \end{enumerate}
\end{lemma}

\begin{proof}
In fact, one can take all simplices congruent to one another. For instance, one can take the (scaled) set of alcoves for the affine Coxeter group $\widetilde A_{n}$ (see~\cite[Chapter~6]{shi1986kazhdan}); this will give an example with a good value of $c_n$ (perhaps, the best). If we are not chasing after good constants, much simpler constructions are possible. One way is to consider the cubical subdivision of $\mathbb{R}^{2n-1}$ with the set of vertices $\eps \mathbb{Z}^{2n-1}$, split each $\eps$-size cube into $(2n-1)!$ simplices, and then take the barycentric subdivision, which can be colored naturally. Either of these constructions can be made periodic easily.
\end{proof}

\begin{definition}
\label{def:join}
Let $X^{2n-1}$ be $\mathbb{R}^{2n-1}$ or $T^n \times \mathbb{R}^{n-1}$. Triangulate it as in Lemma~\ref{lem:triangulation}, and define $Z_i$, $1 \le i \le n$, to be the union of all edges of the triangulation between the vertices of colors $2i-1$ and $2i$. We say that $X$ is the \emph{$\eps$-local join} of $Z_1, \ldots, Z_n$.
\end{definition}

The motivation behind this definition is that every (top-dimensional) simplex $\sigma$ of the triangulation can be written as the join $(\sigma \cap Z_1) * \ldots * (\sigma \cap Z_n)$; that is, any point $x \in \sigma$ can be written as
\[
x = \sum\limits_{i=1}^n t_i z_i, \quad \text{where } z_i \in \sigma \cap Z_i, ~~ t_i \ge 0, ~~ \sum_{i=1}^n t_i = 1.
\]
The coefficients $t_i$ are determined uniquely; if $t_i \neq 0$, the corresponding $z_i$ is determined uniquely too. This defines a map $x \mapsto (t_1, \ldots, t_n)$ from $\sigma$ to the standard $(n-1)$-dimensional simplex $\triangle^{n-1}$; for adjacent simplices of the triangulation, those maps agree on their intersection; hence, we have a well-defined map
\[
\tau : X \to \triangle^{n-1},
\]
which we call the \emph{join map}. Note that $Z_i = \tau^{-1}(v_i)$, where $v_1, \ldots, v_n$ are the vertices of $\triangle^{n-1}$. For each vertex $v_i$, denote the opposite facet of $\triangle^{n-1}$ by $v_i^\vee$. For each complex $Z_i$, introduce its \emph{dual} complex $Z_i^\vee = \tau^{-1}(v_i^\vee)$. In other words, $Z_i^\vee$ is the union of all $(2n-3)$-dimensional cell of our triangulation that do not intersect $Z_i$. There are natural retractions
\[
\pi_i :  X \setminus Z_i^\vee \to Z_i,
\]
defined by sending $x = \sum\limits_{i=1}^n t_i z_i \in \sigma$ to $z_i \in \sigma \cap Z_i$; they are well-defined since $t_i \neq 0$ whenever $x \notin Z_i^\vee$. Note that $\pi_i$ moves each point by distance at most $\sup \diam \sigma \sim \eps$.

\subsection{Manifolds that are locally nearly one-dimensional}

\begin{proof}[Proof of Theorem~\ref{thm:tubes}]

Pick a torus $T^n$ with convexity radius $\ge 1$, as in Lemma~\ref{lem:triangulation}. On scale $\eps$, represent $X = T^n \times \mathbb{R}^{n-1}$ as the local join of one-dimensional complexes $Z_1, \ldots, Z_n$, as in Definition~\ref{def:join}. The goal is to build a manifold $M^n \subset X$ so that on the large scale ($\sim 1$) it resembles $T^n$ homologically, but on the small ($\sim \eps$) scale it will become porous in a way that makes its local $1$-width small.

Recall the join map $\tau : X \to \triangle^{n-1}$ arising from the local join structure of $X$. In this proof, it will be convenient to think of the target simplex as a regular simplex of inradius $3$, placed in $\mathbb{R}^{n-1}$ and centered at the origin. We make use of the join map $\tau : T^n \times \mathbb{R}^{n-1} \to \triangle^{n-1} \subset \mathbb{R}^{2n-1}$ to ``perturb'' the projection $p : T^n \times \mathbb{R}^{n-1} \to \mathbb{R}^{n-1}$ onto the second factor:


\[
\widetilde{p} := p - \tau/2.
\]
The choice of the factor $1/2$ is not particularly important as long as it is less than 1. We only use that the $p$-term dominates the $\tau$-term in the sense that $\widetilde{p}$ does not vanish outside of $T^n \times \interior \triangle^{n-1}$.

Finally, define the ``perturbation of $T^n = p^{-1}(0)$ by the $Z_i=\tau^{-1}(v_i)$'':
\[
M^n := \widetilde{p}^{-1}(0).
\]

Note: as defined, $M$ is a PL-manifold; but we can perturb $\tau$ slightly to make it smooth and to make $0$ a regular value of $\widetilde{p}$; then $M$ becomes a smooth manifold. Observe that $M$ is contained in $T^n \times \interior \triangle^{n-1}$, so $M$ is closed; it is also orientable by construction. See Figure~\ref{fig:tubes} for an illustration of the case $n=2$.

Now, within a unit ball $B_1(x) \subset M$, we want to find a projection on one of the $Z_i$ with $\eps$-small fibers. Recall the notation introduced after Definition~\ref{def:join}: the dual complexes $Z_i^\vee = \tau^{-1}(v_i^\vee)$, and the retractions $\pi_i : X \setminus Z_i^\vee \to Z_i$. One of these retractions would do if we find $i$, depending on $x$, so that $B_1(x) \cap Z_i^\vee = \varnothing$.
Pick $i$ maximizing the distance between $\tau(x)$ and $v_i^\vee$ in $\triangle^{n-1}$; this distance is at least $3$ by our choice of metric on $\triangle^{n-1}$.
When we move $x$ to $x' \in B_1(x)$, its $p$-projection changes by at most $1$, whereas its $\tau$-projection changes by at most $2$ (since the value of $\widetilde{p}$ is fixed), so $\tau(x')$ never reaches $v_i^\vee$. We can now use the retraction $\pi_i : B_1(x) \to Z_i$, showing $\UW_1(B_1(x)) \lesssim \eps$. (Notation $\lesssim$ means inequality that holds up to a factor depending on dimension only.)

To show that $\UW_{n-1}(M) \ge 1$, we use our second tool for estimating widths. Apply Corollary~\ref{cor:secondtool} to the $1$-Lipschitz projection map $M \to T^n$ (the composition $M \hookrightarrow T^n \times \mathbb{R}^{n-1} \to T^n$), sending the fundamental class $[M]\in H_n(M)$ to $[T^n]\neq 0$. Indeed, (the Poincar\'e duals of) the classes of $M$ and $T^n$ are the same in $H^{n-1}(T^n\times \triangle^{n-1}, T^n\times \partial \triangle^{n-1})$ as zero level sets for homotopic mappings $\widetilde{p}$ and $p$, respectively; the homotopy $p-t\tau$, $t \in [0,1/2]$, does not vanish on $T^n\times \partial \triangle^{n-1}$ since the $p$-term dominates the $\tau$-term.
\end{proof}

\begin{remark}
\label{rem:perturbation}
In this argument, we used a torus as the ``base space to be perturbed''. In fact, this construction can be repeated for any reasonable base space, provided that it has sufficient convexity radius. One can easily adapt Lemma~\ref{lem:triangulation} for this case, and the rest of the proof goes unchanged. Morally, the outcome is that any manifold can be ``homologically perturbed'' to make its local (on the scale comparable with its convexity radius) $1$-width arbitrarily small.
\end{remark}

\begin{remark}
\label{rem:scaling}
The parameters of the construction can be adjusted in order to get a manifold $M^n$ of $\UW_{n-1}(M) \gtrsim \beta^{1/n}$, with all unit balls $B \subset M$ having $\UW_1(B) \lesssim 1$. Here $\beta$ is the first Betti number of $M$. The adjustment is to start from a torus of convexity radius $\sim \beta^{1/n}$, and pick the triangulation scale $\eps \sim 1$.
\end{remark}

\subsection{Topologically simple \texorpdfstring{$n$}{n}-manifolds that are locally nearly \texorpdfstring{$\log n$}{log n}-dimensional}

The next result is an ``amplification'' of Guth's example~\cite[Section~4]{guth2017volumes} of a $3$-sphere with large $2$-width but all unit balls $\UW_2$-small. We start by taking a reasonable base space $X^{2n-1}$ equipped with an $\eps$-fine triangulation, as in Lemma~\ref{lem:triangulation}. A particular choice of $X$ is not really important; the only assumptions we need are its substantial codimension $1$ width, and the existence of a colored triangulation (local join representation). For example, one can take $X$ to be a unit cube in the Euclidean space $\mathbb{R}^{2n-1}$; then $\UW_{2n-2}(X) \ge 1$ by the first tool. If one takes a unit Euclidean ball, its codimension $1$ width is known exactly (see~\cite[Remark~6.10]{akopyan2012borsuk}), but one can use the second tool to get a weaker bound $\ge 1/2$; then one can take $X$ to be the ball of radius $2$ in order to have $\UW_{2n-2}(X) \ge 1$. In both examples, a local join structure is given by Lemma~\ref{lem:triangulation}. It is easy to modify the argument in order to take $X$ a sphere, or a torus, etc.

The triangulation of $X$ comes equipped with one-dimensional complexes $Z_1, \ldots, Z_n$, and the join map $\tau: X \to \triangle^{n-1}$ mapping $Z_i$ to the $i\textsuperscript{th}$ vertex of the simplex. Now we blow up the metric in $X$ along the $\triangle^{n-1}$-direction and leave it unchanged along the fibers of $\tau$. Formally speaking, endow $\triangle^{n-1}$ with an auxiliary metric making $\triangle^{n-1}$ a regular simplex with inradius $2$ (this choice will be explained later), and add its pullback to the metric of $X$. The resulting metric on $X$ is piecewise Riemannian, and after a slight smoothening, we get a Riemannian manifold $X'$.

\begin{proposition}
\label{prop:ball1}
$\UW_{2n-2}(X') \ge 1$ but $\UW_{n}(B) \lesssim \eps$ for every unit ball $B \subset X'$.
\end{proposition}

\begin{proof}
By assumption, $\UW_{2n-2}(X) \ge 1$. The metric on $X'$ is even larger, so $\UW_{2n-2}(X') \ge 1$.

Now take a unit ball $B \subset X'$, and observe that $\UW_{n}(B) \le \UW_{n}(\tau^{-1}(\tau(B)))$. By construction of the blown-up metric of $X'$ (namely, by the choice of the auxiliary metric on $\triangle^{n-1}$), $\tau(B)$ misses at least one facet of $\triangle^{n-1}$, say, the $i\textsuperscript{th}$ one. Then the retraction $\pi_i$ (in the notation introduced after Definition~\ref{def:join}) gives a map $\tau^{-1}(\tau(B)) \to Z_i$, whose fibers are small in the original metric of $X$. The map
\begin{align*}
  B &\to Z_i \times \triangle^{n-1} \\
  x &\mapsto (\pi_i(x), \tau(x))
\end{align*}
gives a desired bound on $\UW_{n}(B)$.
\end{proof}

For $n=2$ this construction recovers Guth's example. Let us rephrase this construction once again, since we are going to apply it inductively.

\textit{The blow-up construction.} Start from a manifold $X$ with metric $g_X$, and a piecewise smooth join map $\tau : X \to \triangle$ (obtained from a fine colored triangulation of $X$). Suppose $\triangle$ is equipped with an auxiliary metric $g_{\triangle}$, in which no unit ball $B^\triangle \subset \triangle$ meets all facets of $\triangle$. Consider the piecewise Riemannian metric $g_X + \tau^* g_\triangle$, and perturb it slightly to get a smooth metric $g_X'$.
We say that $g_X'$ is obtained from $g_X$ by \emph{blowing it up along the $\triangle$-direction, or across the join map}, via the auxiliary metric $g_\triangle$.

Observe that the distances in the blown-up metric $g_X'$ do not decrease in comparison with the original metric $g_X$, hence the Urysohn width does not decrease either. By the same reason, if the space $X$ itself was a simplex with the property ``no unit ball meets all facets'', the same holds true after the blow-up.

In the blow-up in Guth's example, the auxiliary metric on $\triangle^{n-1}$ was Euclidean, making it a regular simplex of inradius $2$. There is a more clever way to pick this auxiliary metric to get better estimates. 
Let $n = 2m$. We will start from a metric making $\triangle^{n-1}$ a regular simplex with inradius 2, and blow it up with the goal to have every unit ball $B^\triangle$ in $\triangle^{n-1}$ small in the sense of some width. Repeat the construction above: pick $m$ skeleta, each of dimension $1$, in a fine colored triangulation inside $\triangle^{n-1}$, and blow up the metric of $\triangle^{n-1}$ across the join map in order to have maps $B^\triangle \to Y^{m}$ with small fibers, for every unit ball $B^\triangle \subset \triangle^{n-1}$ (this is the conclusion of Proposition~\ref{prop:ball1}). Now, using the modified metric on $\triangle^{n-1}$, we blow up the metric of $X$ along the $\triangle^{n-1}$-direction. Call the resulting metric space $X''$.

\begin{proposition}
\label{prop:ball2}
$\UW_{4m-2}(X'') \ge 1$ but $\UW_{m+1}(B) \lesssim \eps$ for every unit ball $B \subset X''$.
\end{proposition}

\begin{proof}
$\UW_{m+1}(B) \le \UW_{m+1}(\tau^{-1}(\tau(B)))$, where $\tau(B)$ lies in a unit ball $B^\triangle \subset \triangle^{n-1}$, missing, say, the $i\textsuperscript{th}$ facet of $\triangle^{n-1}$. Then there is a map $\tau^{-1}(\tau(B)) \to Z_i \times Y^{m}$ with small fibers, defined as follows: a point $x$ gets mapped to $(\pi_i(x), y) \in Z_i \times Y^{m}$, where $y$ is the image of $\tau(x)$ under the map $B^\triangle \to Y^{m}$.
\end{proof}

Iterating this procedure $\ell$ times, we arrive at the following conclusion.

\begin{proposition}
\label{prop:ball3}
For a unit Euclidean cube $X$ (or a regular simplex of inradius $2$, or a ball, or a sphere, or a torus) of dimension $2^\ell k - 1$, there is a way to blow up the metric in order to get a space $X^{(\ell)}$ such that $\UW_{2^\ell k - 2}(X^{(\ell)}) \ge 1$ but $\UW_{k+\ell-1}(B) \lesssim \eps$ for every unit ball $B \subset X^{(\ell)}$.
\end{proposition}

\begin{proof}
The original metric on $X$ satisfies $\UW_{2^\ell k -2}(X) \ge 1$, and an $\eps$-local join structure on $X$ gives the join map $\tau: X \to \triangle^{2^{\ell-1}k - 1}$. We blow up the metric of $X$ across $\tau$ using a carefully chosen auxiliary metric on $\triangle^{2^{\ell-1}k - 1}$. Inducting on $\ell$, we may assume that there is a metric on $\triangle^{2^{\ell-1}k - 1}$ such that no unit ball meets all its facets, and every unit ball is small in the sense of $\UW_{k+(\ell-1)-1}$. We use this metric to blow up the metric of $X$ and this way get $X^{(\ell)}$. Arguing as in the proof of Proposition~\ref{prop:ball2}, one makes sure that every unit ball of $X^{(\ell)}$ is small in the sense of $\UW_{k+\ell-1}$.
\end{proof}

\begin{proof}[Proof of Theorem~\ref{thm:ball}]
Let $\ell = \lceil \log_2 (n+1) \rceil$, and apply Proposition~\ref{prop:ball3} with $k=1$ to get a ball $X$ (or a sphere, or a torus) of dimension $2^\ell - 1$ with large global $(2^\ell - 2)$-width but small local $\ell$-width. Now we can build $M^n$ as a subspace of $X$.
\end{proof}

\section{Open problems}

\begin{question}
\label{ques:surface}
Let $M^2$ be a closed Riemannian surface with the first $\mathbb{Z}/2$-Betti number $\beta$, and with every unit ball having $1$-width less than $\eps$, for some fixed small absolute constant $\eps$. In the optimal bound $\UW_1(M) \lesssim f(\beta)$, what is the order of magnitude of the right hand side? It must be between $\beta^{1/2}$ (by Remark~\ref{rem:scaling}) and $\beta$ (by Theorem~\ref{thm:surface}).
\end{question}

\begin{question}
\label{ques:ball}
Let $M^n$ be a Riemannian $n$-sphere, $n\ge 4$, with every unit ball having $d$-width less than $\eps_n$, for some fixed small dimensional constant $\eps_n$. What is the smallest $d = d(n)$ such that the assumption on local width would imply $\UW_{n-1}(M) \lesssim 1$? Theorems~\ref{thm:surface}~and~\ref{thm:ball} imply $2 \le d(n) < \log_2 (n+1)$.
\end{question}

\bibliography{ref}
\bibliographystyle{abbrv}
\end{document}